\documentclass[12pt, english, a4paper]{article}
\usepackage[latin9]{inputenc}
\usepackage[T1]{fontenc}
\usepackage{graphicx}
\usepackage{amsthm}
\usepackage{amsfonts}
\usepackage{mathtools}
    \allowdisplaybreaks
\usepackage{amssymb}
\usepackage{makecell}
\usepackage{bm}
\usepackage{comment}
\usepackage{cases}
\usepackage{hyperref}
\usepackage{mathrsfs}
\usepackage{fullpage}
\usepackage{times}
\usepackage[usenames]{color}
\usepackage[dvipsnames]{xcolor}
\usepackage{hyperref}
\usepackage{tikz,tikz-cd}
\usepackage{extarrows}
\usepackage{float}
\usepackage{ytableau}
\usepackage{xcolor}
\usepackage{circledsteps}
 \usepackage{multirow}
\definecolor{lgrey}{rgb}{0.8,0.8,0.8 }
\definecolor{dgrey}{rgb}{0.3,0.3,0.3 }

\newcommand{\red}[1]{\textcolor{red}{#1}}

\newcommand{\card}[1]{\lvert #1 \rvert}

\theoremstyle{plain}
\newtheorem{theorem}{Theorem}
\newtheorem{note}{Note}
\newtheorem{corollary}[theorem]{Corollary}
\newtheorem{lemma}[theorem]{Lemma}

\theoremstyle{definition}

\theoremstyle{remark}

\usepackage{authblk}

\usepackage{ytableau}
\usepackage{xcolor}
\definecolor{lgrey}{rgb}{0.8,0.8,0.8 }
\definecolor{dgrey}{rgb}{0.3,0.3,0.3 }

\title{Families of Shape-Wilf-Equivalent Claw-Shaped Partially Ordered Patterns}

\author{Sucharita Biswas
\thanks{{biswas.sucharita56@gmail.com}}}
\affil{Department of Mathematics, Indian Institute of Technology, Bombay, Mumbai 400076, India} 
\begin{document}

\maketitle
\begin{abstract}
Partially ordered patterns (POPs) generalize classical permutation patterns and have been extensively studied in the contexts of permutations, words, compositions, and partitions. Burstein, Han, Kitaev, and Zhang established the shape-Wilf-equivalence for individual claw-shaped POPs. In this paper, we extend their result by proving that certain families of claw-shaped POPs are shape-Wilf-equivalent and enumerate the number of permutations avoiding that set of claw-shaped POPs. Our approach is based on a new encoding process, which is entirely different from the method used in their work.
\end{abstract}
\textbf{\small{Keyword:}}{\small{} permutation pattern, partially ordered pattern, shape-Wilf-equivalence, Wilf-equivalence }{\let\thefootnote\relax}
\section{Introduction}
Let $\mathbb{N}$ denote the set of all natural numbers. For $n \in \mathbb{N}$ let $[n]=\{1,2,\ldots , n\}$ and $S_n$ be the set of all permutations of the set $[n]$. An occurrence of a classical pattern of size $k$, $p=p_1p_2\cdots p_k$ in a permutation $\pi$ is a subsequence $\pi_{i_1}\dots\pi_{i_k}$ where $1\leq i_1<i_2<\cdots <i_k\leq n$ such that $\pi_{i_j}<\pi_{i_m}$ if and only if $p_j<p_m$. For example the permutation $45213$ has $3$ occurrences of  the pattern $231$ namely the subsequences $452$, $451$ and $453$. Pattern in permutations is a well studied topic in combinatorics with many applications in computer science, computational biology cryptography and many other areas. Not only in permutations, patterns are also studied in words. For further references see the book by Kitaev \cite{kitaev-patterns-words}. 

A partially ordered pattern (POP) is a generalization of classical pattern. A \emph{partially ordered pattern} (POP) $p$ of size $k$ is defined by a $k$-element partially ordered set (poset) $P$ labeled by the elements in $\{1,\dots, k\}$. An occurrence of such a POP $p$ in a permutation $\pi = \pi_1\dots\pi_n$ is a subsequence $\pi_{i_1}\dots \pi_{i_k}$, where $1 \le i_1 <\dots < i_k \le n$, such that for $1\leq j < m \leq k$, $\pi_{i_j} < \pi_{i_m}$ if and only if $j < m$ in $P$. Thus, a classical pattern of length $k$ corresponds to a $k$-element chain. For example consider the POP $p=\begin{tikzpicture}[baseline=(current bounding box)]
		\filldraw  (0,.2) circle (2pt);
        \filldraw  (-.2,-.2) circle (2pt);
        \filldraw  (.2,-.2) circle (2pt);
        \filldraw  (.6,0) circle (2pt);
		
		\draw (0,.2) -- (-.2,-.2); 
        \draw (0,0.2) -- (0.2,-.2); 
        \node[] at (0,.5)   { $1$};
		\node[] at (-.3,-.5)  { $2$};
		\node[] at (.3,-.5) { $4$};
        \node[] at (.8,0)   { $3$};
        
\end{tikzpicture}$. Here $1>2$, $1>4$ and $2,4$ are non comparable and no condition on $3$. Hence $p=\{3 1 4 2, 3 2 4 1, 4 1 2 3, 4 1 3 2, 4 2 1 3, 4 2 3 1, 4 3 1 2, 4 3 2 1\}$.
The permutation $45213$ has $2$ occurrences of $p$ namely, $4213$, $5213$. Hence avoiding the POP $p$ means avoiding all the classical patterns $3 1 4 2, 3 2 4 1, 4 1 2 3, 4 1 3 2, 4 2 1 3, 4 2 3 1, 4 3 1 2, 4 3 2 1$ at a time. 
POPs were introduced by Kitaev \cite{kitaev-pop} and studied  extensively in \cite{gao-kitaev-length-4-5-pops},  \cite{pop-comp},  \cite{segmented-pop}, \cite{introduction-pop}, \cite{survey-pop}, \cite{pop-k-ary}, \cite{v-lambda}.
 We will use the word ``pattern'' to mean a classical pattern and ``POP" to mean partially ordered pattern. 

 A permutation is said to avoid a pattern $p$ if it has no occurrence of $p$. Similarly a permutation is said to avoid a set of patterns $P$  if it does not contain any of those patterns from the set $P$. Set of all permutation of length $n$ avoiding the pattern $p$ is denoted by $S_n(p)$, similarly avoiding the set of patterns $P$ is denoted by $S_n(P)$. Two sets of patterns $P_1$ and $P_2$ are said to be \textit{Wilf-equivalent} if the number of permutations of length $n$ avoiding each pattern in $P_1$ is equal to that avoiding each pattern in $P_2$ for all $n\geq 1$, i.e., $|S_n(P_1)|=|S_n(P_2)|$ for all $n\geq 1$ and it is denoted by $P_1\sim P_2.$ If $\{p_1\}\sim \{p_2\}$ for simplicity we write $p_1\sim p_2$.

The study of shape-Wilf-equivalence is naturally facilitated by examining pattern-avoiding transversals of Ferrers boards, which serve as a generalization of classical permutation matrices. We establish the following conventions and terminology. Ferrers boards are drawn using French notation, where row lengths increase from top to bottom. Columns are indexed from left to right and rows from top to bottom, with $(i,j)$ denoting the cell located in row $i$ and column $j$.

A transversal of a Ferrers board $\lambda=(\lambda_1,\dots,\lambda_n)$ is a $0$--$1$ filling of the board in which every row and every column contains exactly one $1$. We let $S_\lambda$ denote the set of all transversals of $\lambda$. In this context, a permutation $\pi=\pi_1\pi_2\dots\pi_n$ naturally corresponds to a transversal of an $n \times n$ square board: the cells $(\pi_i,i)$ are filled with $1$ for all $1 \le i \le n$, and all remaining cells are filled with $0$. This specific transversal represents the permutation matrix of $\pi$.

Let $\alpha \in S_k$ be a permutation with corresponding permutation matrix $M_\alpha$. A transversal $T$ of an $n \times n$ Ferrers board $\lambda$ is said to contain the pattern $\alpha$ if there exist index subsets of rows $R = \{r_1, \dots, r_k\}$ and columns $C = \{c_1, \dots, c_k\}$ in $[n]$ such that the submatrix $M'$ induced by $R$ and $C$ is a copy of $M_\alpha$, and every cell $(r_i, c_j)$ lies within $\lambda$. In this case, we say $M'$ is an occurrence of $\alpha$, or that $M'$ is isomorphic to $\alpha$. If no such occurrence exists, $T$ avoids $\alpha$, and we call $T$ $\alpha$-avoiding.

For a set of patterns $P$, let $S_\lambda(P)$ denote the set of transversals of $\lambda$ that avoid every pattern in $P$. Two sets of patterns $P_1$ and $P_2$ are said to be shape-Wilf-equivalent, denoted $P_1 \sim_s P_2$, if $\card{S_\lambda(P_1)} = \card{S_\lambda(P_2)}$ for every Ferrers board $\lambda$. When dealing with singleton sets $P=\{\sigma\}$ and $Q=\{\tau\}$, we drop the braces and simply write $\sigma \sim_s \tau$. By definition, shape-Wilf-equivalence naturally implies standard Wilf-equivalence.

In \cite{burstein-shape-wilf}, the authors investigated the shape-Wilf-equivalence of partially ordered patterns (POPs), proving that any two POPs of the form shown in Figure \ref{Fig1} are shape-Wilf-equivalent using a technique inspired by \cite{bwx-main}. In this paper, we generalize their result. Our main result, Theorem \ref{main_thm}, establishes this generalization using a fundamentally different proof technique. Extending an approach introduced in \cite{SUK-arxiv}, we develop an encoding scheme that maps a transversal of a Ferrers board $\lambda$ avoiding a set of POPs $P$ to a specific word over the alphabet $\{0, 1, \dots, n\}$. We utilize this encoding to construct a bijection between $S_{\lambda}(P^a_{(m,k,d)})$ and $S_{\lambda}(P^b_{(m,k,d)})$ for any Ferrers board $\lambda$, where the sets $P^a_{(m,k,d)}$ and $P^b_{(m,k,d)}$ are defined below. Furthermore, in \cite{gao-kitaev-length-4-5-pops}  Gao and Kitaev enumerated the permutations avoiding the specific POPs in Figure \ref{Fig1}, in this paper we provide a generalized formula to count the permutations avoiding the broader set of POPs $P^a_{(m,k,d)}$.

The remainder of this paper is organized as follows. First, we formally introduce the specific families of claw-shaped POPs under consideration and detail our novel, recursive encoding process for transversals of Ferrers boards. Next, we rigorously establish the conditions for valid cell placements during this insertion process and leverage our encoding scheme to prove our main result, demonstrating the shape-Wilf-equivalence of these pattern families. Finally, we present a generalized formula to enumerate the permutations that avoid this broader set of POPs. 
\begin{figure}[!htbp]

   \centering
{
\begin{tikzpicture}[baseline=(current bounding box),scale=1]
     \filldraw  (0,0) circle (2pt);
    \filldraw  (-1.25,-1.5) circle (2pt);
    \filldraw  (-.5,-1.5) circle (2pt);
    \filldraw  (.25,-1.5) circle (2pt);
    \filldraw  (1.25,-1.5) circle (2pt);
    
    \node[] at (.75,-1.5)   {  \Large$\ldots$};

     \draw [thick](0,0) -- (-1.25,-1.5);
     \draw [thick](0,0) -- (-.5,-1.5);
     \draw [thick](0,0) -- (.25,-1.5);
     \draw [thick](0,0) -- (1.25,-1.5);

    \node[] at (0,.3)   {$x_1$};
     \node[] at (-1.25,-1.8)   { $x_2$};
    \node[] at (-.5,-1.8)   {$x_3$};
    \node[] at (.25,-1.8)   {$x_4$};
    \node[] at (1.3,-1.8)   {$x_m$};
\end{tikzpicture}
}
\caption{Claw-shaped POP}  
\label{Fig1}
 \end{figure}
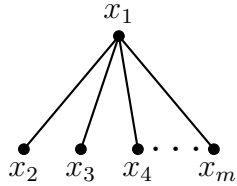

\begin{theorem}{(\cite{burstein-shape-wilf}, Theorem 3)}\label{p1-sim-p2}
Let $p_1$ and $p_2$ be any two POPs as in Figure \ref{Fig1}, where $\{x_1, x_2, \ldots , x_k\} =
\{1, 2, \ldots , k\}$ and $k\geq 1$. Then $p_1 \sim_s p_2$.
\end{theorem}

\begin{theorem}{(\cite{gao-kitaev-length-4-5-pops}, Theorem 2)}\label{claw_size}
    Let $p$ be the POP in Figure \ref{Fig1}, where $\{x_1,\ldots , x_m\}=\{1, \ldots ,m\}$ and $m\geq 1$. Then 
    $|S_n(p)|=\begin{cases}
        n! & \text{ if }n<m,\\
        (m-1)!(m-1)^{n-m+1} &\text{ if }n \geq m.
    \end{cases}$
\end{theorem}

We now introduce the families of POPs for which we establish shape-Wilf equivalence.

Let $k,d \in \mathbb{N}\cup \{0\}$, $a,m\in \mathbb{N}$ such that $a+kd\leq m$. Consider the set of partial order patterns (POPs), 

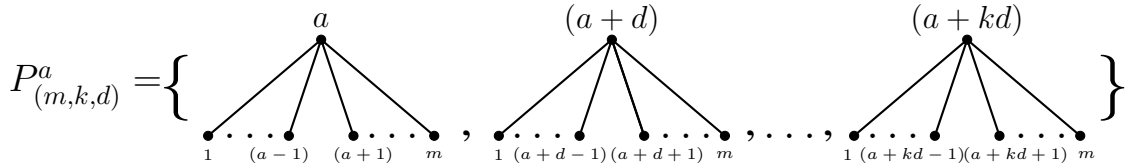
\begin{tikzpicture}[scale=.85]
    \node[] at (-3.5,-.75)   {\large$P^a_{(m,k,d)}=~~$};
    \node[] at (-2.25,-.75)   {\Huge$\lbrace$};

    \filldraw  (0,0) circle (2pt);
    \filldraw  (-1.75,-1.5) circle (2pt);
    \filldraw  (-.5,-1.5) circle (2pt);
    \filldraw  (.5,-1.5) circle (2pt);
    \filldraw  (1.75,-1.5) circle (2pt);
    \node[] at (-1.1,-1.5)   {  \Large$\ldots$};
    \node[] at (1.1,-1.5)   {  \Large$\ldots$};

     \draw [thick](0,0) -- (-1.75,-1.5);
     \draw [thick](0,0) -- (-.5,-1.5);
     \draw [thick](0,0) -- (.5,-1.5);
     \draw [thick](0,0) -- (1.75,-1.5);

    \node[] at (0,.3)   {$a$};
     \node[] at (-1.75,-1.8)   {\tiny $1$};
    \node[] at (-.65,-1.8)   {\tiny$(a-1)$};
    \node[] at (.65,-1.8)   {\tiny$(a+1)$};
    \node[] at (1.75,-1.8)   {\tiny$m$};

    \node[] at (2.25,-1.5)   {\Large,};

    \filldraw  (4.5,0) circle (2pt);
    \filldraw  (2.75,-1.5) circle (2pt);
    \filldraw  (4,-1.5) circle (2pt);
    \filldraw  (5,-1.5) circle (2pt);
    \filldraw  (6.25,-1.5) circle (2pt);
    \node[] at (3.35,-1.5)   {  \Large$\ldots$};
    \node[] at (5.65,-1.5)   {  \Large$\ldots$};

    \node[] at (6.65,-1.5)   {\Large,};

    \draw [thick](4.5,0) -- (2.75,-1.5);
    \draw [thick](4.5,0) -- (4,-1.5);
    \draw [thick](4.5,0) -- (5,-1.5);
   \draw [thick](4.5,0) -- (5,-1.5);
   \draw [thick](4.5,0) -- (6.25,-1.5);

    \node[] at (4.5,.3)   {$(a+d)$};
     \node[] at (2.75,-1.8)   {\tiny$1$};
    \node[] at (3.7,-1.8)   {\tiny$(a+d-1)$};
    \node[] at (5.2,-1.8)   {\tiny$(a+d+1)$};
    \node[] at (6.25,-1.8)   {\tiny$m$};

    \node[] at (7.25,-1.5)   {  \Large$\ldots$};
    \node[] at (7.85,-1.5)   {\Large,};

    \filldraw  (10,0) circle (2pt);
    \filldraw  (8.25,-1.5) circle (2pt);
    \filldraw  (9.5,-1.5) circle (2pt);
    \filldraw  (10.5,-1.5) circle (2pt);
    \filldraw  (11.75,-1.5) circle (2pt);

    \node[] at (10,.3)   {$(a+kd)$};
     \node[] at (8.15,-1.8)   {\tiny$1$};
    \node[] at (9.15,-1.8)   {\tiny$(a+kd-1)$};
    \node[] at (10.75,-1.8)   {\tiny$(a+kd+1)$};
    \node[] at (11.85,-1.8)   {\tiny$m$};

    \node[] at (8.85,-1.5)   {  \Large$\ldots$};
    \node[] at (11.15,-1.5)   {  \Large$\ldots$};

    \draw [thick](10,0) -- (8.25,-1.5);
    \draw [thick](10,0) -- (9.5,-1.5);
    \draw [thick](10,0) -- (10.5,-1.5);
    \draw [thick](10,0) -- (11.75,-1.5);

    \node[] at (12.25,-.75)   {\Huge$\rbrace$};
\end{tikzpicture}
\begin{theorem}\label{main_thm}
    For any distinct $a,b\in [m]$, $P^a_{(m,k,d)}\sim_s P^b_{(m,k,d)}.$
\end{theorem}
\begin{note}
    For shape Wilf-equivalence, the parameters $m$, $k$, and $d$ must be the same for both sets $P^a_{(m,k,d)}$ and $P^b_{(m,k,d)}$. If $m$, $k$ or $d$ differ, then the sets are not shape Wilf-equivalent. For example,
\[
P^1_{(3,0,0)} \nsim_s P^2_{(4,0,0)}, 
\qquad 
P^1_{(4,1,1)} \nsim_s P^2_{(4,1,2)}, 
\qquad 
P^1_{(4,0,0)} \nsim_s P^2_{(4,1,2)}.
\]
\end{note}
\begin{note}
    When $P^a_{(m,k,d)}$ is a singleton set then we consider $k=d=0$. Therefore $P^{x_1}_{(m,0,0)}=
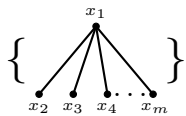
\begin{tikzpicture}[baseline=(current bounding box),scale=.6]
    \node[] at (-1.75,-.75)   {\Large$\lbrace$};
     \filldraw  (0,0) circle (2pt);
    \filldraw  (-1.25,-1.5) circle (2pt);
    \filldraw  (-.5,-1.5) circle (2pt);
    \filldraw  (.25,-1.5) circle (2pt);
    \filldraw  (1.25,-1.5) circle (2pt);
    \node[] at (1.75,-.75)   {\Large$\rbrace$};
    \node[] at (.75,-1.5)   { \small $\ldots$};

     \draw [thick](0,0) -- (-1.25,-1.5);
     \draw [thick](0,0) -- (-.5,-1.5);
     \draw [thick](0,0) -- (.25,-1.5);
     \draw [thick](0,0) -- (1.25,-1.5);

    \node[] at (0,.3)   {\tiny$x_1$};
     \node[] at (-1.25,-1.8)   { \tiny$x_2$};
    \node[] at (-.5,-1.8)   {\tiny$x_3$};
    \node[] at (.25,-1.8)   {\tiny$x_4$};
    \node[] at (1.3,-1.8)   {\tiny$x_m$};
\end{tikzpicture}.$
    Hence Theorem \ref{main_thm} is generalization of Theorem \ref{p1-sim-p2}.
\end{note}
At first we will define the encoding process explicitly. 
\paragraph{ Encoding Process:}
To derive the encoding word associated with a transversal, we use the following recursive procedure. First, all entries equal to 1 are removed from the Ferrers board. They are then reinserted in their original positions, proceeding row by row from the topmost row downward.
\begin{enumerate}\label{process}
\item[Step 1:] Initialize an empty Ferrers board with all cells colored white.
\item[Step 2:] Restore the $1$ belonging to the topmost row to its original position, then color its row and column gray.
\item[Step 3:] Assuming the $1$'s for the top $k$ rows have been placed, position the next $1$ in the topmost row of the remaining white subboard. Color its corresponding row and column gray.
\end{enumerate}

 When the entries are inserted according to the above procedure to obtain a transversal avoiding \(P^a_{(m,k,d)}\) or \(P^b_{(m,k,d)}\), not every position in a row is admissible. Indeed, placing a \(1\) in certain cells would inevitably force an occurrence of the POPs in combination with either previously placed entries or entries that must be placed later.
We call a cell (or position) \emph{valid} if placing a \(1\) in that cell does not create an occurrence of \(P^a_{(m,k,d)}\) or \(P^b_{(m,k,d)}\), either with respect to the entries already placed or with respect to any possible future placements. In other word, a cell is valid iff there exists at least one completion of the remaining white subboard into a transversal avoiding the POP family.

\begin{lemma}
    Let $\lambda$ be the Ferrers board. At some stage $i$ of the respective insertion processes for $P^a_{(m,k,d)}$, let $\lambda'$ denote the board consisting of the remaining white cells of $\lambda$. Suppose that the top row of $\lambda'$ corresponds to row $r$ and the columns $c_1<\dots<c_l$ of $\lambda$.
    If $l\geq m$ then $I$ is the set of all valid position where $1$ can be placed while still avoiding the pattern set $P^a_{(m,k,d)}$, where
$$ I=[a-1]\cup 
\{\, i\in[l]\setminus\{a+jd:0\le j\le k\} :
a+jd<i \Rightarrow l-i<m-(a+jd),\;\forall ~0\le j\le k \,\}.
$$
\end{lemma}
\begin{proof}
Let \(i\) be a position in the topmost row. We first characterize when \(i\) is valid.

If \(i \in [a-1]\), then placing a \(1\) in column \(c_i\) cannot create an occurrence of any POP from \(P^a_{(m,k,d)}\), and hence \(i\) is valid.

Next, observe that \(i \neq a+jd\) for all \(0 \le j \le k\). Consider a position \(i > a+jd\) for some \(0 \le j \le k\). If
\[
l - i \ge m - (a+jd),
\quad \text{i.e.,} \quad
i + m - (a+jd) \le l,
\]
then the columns
$$
c_{\,i-(a+jd)+1}, \ldots, c_i, \ldots, c_{\,i+m-(a+jd)}
$$
form an occurrence of the POP
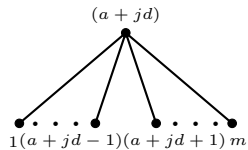
\begin{tikzpicture}[baseline=(current bounding box),scale=0.8]
    \filldraw  (10,0) circle (2pt);
    \filldraw  (8.25,-1.5) circle (2pt);
    \filldraw  (9.5,-1.5) circle (2pt);
    \filldraw  (10.5,-1.5) circle (2pt);
    \filldraw  (11.75,-1.5) circle (2pt);

    \node at (10,.3)   {\tiny$(a+jd)$};
    \node at (8.15,-1.8)   {\tiny$1$};
    \node at (9.1,-1.8)   {\tiny$(a+jd-1)$};
    \node at (10.8,-1.8)   {\tiny$(a+jd+1)$};
    \node at (11.85,-1.8)   {\tiny$m$};

    \node at (8.85,-1.5) {\Large$\ldots$};
    \node at (11.15,-1.5) {\Large$\ldots$};

    \draw[thick](10,0) -- (8.25,-1.5);
    \draw[thick](10,0) -- (9.5,-1.5);
    \draw[thick](10,0) -- (10.5,-1.5);
    \draw[thick](10,0) -- (11.75,-1.5);
\end{tikzpicture}
which belongs to \(P^a_{(m,k,d)}\), a contradiction. Therefore,
\[
l - i < m - (a+jd).
\]

Conversely, assume that
\[
l - i < m - (a+jd) \quad \text{for all } 0 \le j \le k.
\]
Suppose, for contradiction, that \(i\) is not a valid position. Then \(c_i\), together with \(m-1\) other columns, forms an occurrence of a POP from \(P^a_{(m,k,d)}\), say
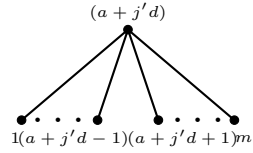
\begin{tikzpicture}[baseline=(current bounding box),scale=0.8]
    \filldraw  (10,0) circle (2pt);
    \filldraw  (8.25,-1.5) circle (2pt);
    \filldraw  (9.5,-1.5) circle (2pt);
    \filldraw  (10.5,-1.5) circle (2pt);
    \filldraw  (11.75,-1.5) circle (2pt);

    \node at (10,.3)   {\tiny$(a+j'd)$};
    \node at (8.15,-1.8)   {\tiny$1$};
    \node at (9.1,-1.8)   {\tiny$(a+j'd-1)$};
    \node at (10.9,-1.8)   {\tiny$(a+j'd+1)$};
    \node at (11.9,-1.8)   {\tiny$m$};

    \node at (8.85,-1.5) {\Large$\ldots$};
    \node at (11.15,-1.5) {\Large$\ldots$};

    \draw[thick](10,0) -- (8.25,-1.5);
    \draw[thick](10,0) -- (9.5,-1.5);
    \draw[thick](10,0) -- (10.5,-1.5);
    \draw[thick](10,0) -- (11.75,-1.5);
\end{tikzpicture}
for some \(0 \le j' \le k\). However, forming such a POP requires at least \(m - (a+j'd)\) columns to the right of \(c_i\), which is impossible since
\[
l < i + m - (a+j'd).
\]
This contradiction shows that \(i\) must be valid, this completes the proof.
\end{proof}

\begin{note}
    If \(l<m\), then all \(l\) columns in the topmost row are valid for placing \(1\), since this placement cannot create an occurrence of any POP from \(P^a_{(m,k,d)}\).
\end{note}

Using the above scheme, we create an encoding that maps transversals that avoid $P^a_{(m,k,d)}$ to those that avoid $P^b_{(m,k,d)}$. The encoding will be a word with the letters $\{0,1,\ldots , n\}$, obtained in the following way. Let $T$ be such a transversal of $F$, and let $w = w_1 \dots w_n$ be the associated word over $\{0,1,\ldots,n\}$. Remove all the $1$'s and place them back on the board according the process described above in \ref{process}. For the $1$ in the $i$-th row (from the top), set:  
\begin{enumerate}
 \item $w_i=j$ if it is in the $j$-th valid position (from left to right) of the subboard of remaining white cells and its placement is not forced, i.e., $|I|\geq 2$. 
    \item $w_i = 0$ if the placement is forced, i.e., $|I| = 1$. 
\end{enumerate}
Now, to obtain a transversal of $F$ that avoids $P^b_{(m,k,d)}$, we invert the encoding as follows. Let $w=w_1\dots w_n$ be the encoding word.
\begin{enumerate}
    \item If $w_i=j$, place $1$  in the $j$-th valid position (from left to right) of the subboard of the remaining white cells.
    \item If $w_i=0$, place it in the only suitable position available in the top row of white cells.
    \item Return the transversal after all the $1$'s are placed.
\end{enumerate}

We illustrate the encoding procedure with an example. Consider the pattern sets 
$P^1_{(4,1,2)}$= \begin{tikzpicture}[baseline=(current bounding box), scale=.8]
    \node[] at (-1.25,-.75)   {\Huge$\lbrace$};

    \filldraw  (0,0) circle (2pt);
    \filldraw  (-.75,-1.5) circle (2pt);
    \filldraw  (0,-1.5) circle (2pt);
    \filldraw  (.75,-1.5) circle (2pt);

     \draw [thick](0,0) -- (-.75,-1.5);
     \draw [thick](0,0) -- (0,-1.5);
     \draw [thick](0,0) -- (.75,-1.5);
     
    \node[] at (0,.3)   {\small$1$};
     \node[] at (-.75,-1.8)   {\small$2$};
    \node[] at (0,-1.8)   {\small$3$};
    \node[] at (.75,-1.8)   {\small$4$};

    \node[] at (1.25,-1.5)   {\Large,};

    \filldraw  (2.5,0) circle (2pt);
    \filldraw  (1.75,-1.5) circle (2pt);
    \filldraw  (2.5,-1.5) circle (2pt);
    \filldraw  (3.25,-1.5) circle (2pt);
   
    \node[] at (3.75,-.75)   {\Huge$\rbrace$};

    \draw [thick](2.5,0) -- (1.75,-1.5);
    \draw [thick](2.5,0) -- (2.5,-1.5);
    \draw [thick](2.5,0) -- (3.25,-1.5);
   
    \node[] at (2.5,.3)   {\small$3$};
     \node[] at (1.75,-1.8)   {\small$1$};
    \node[] at (2.5,-1.8)   {\small$2$};
    \node[] at (3.25,-1.8)   {\small$4$};
    
\end{tikzpicture} and 
 $P^2_{(4,1,2)}$= \begin{tikzpicture}[baseline=(current bounding box), scale=.8]
    \node[] at (-1.25,-.75)   {\Huge$\lbrace$};

    \filldraw  (0,0) circle (2pt);
    \filldraw  (-.75,-1.5) circle (2pt);
    \filldraw  (0,-1.5) circle (2pt);
    \filldraw  (.75,-1.5) circle (2pt);

     \draw [thick](0,0) -- (-.75,-1.5);
     \draw [thick](0,0) -- (0,-1.5);
     \draw [thick](0,0) -- (.75,-1.5);
     
    \node[] at (0,.3)   {\small$2$};
     \node[] at (-.75,-1.8)   {\small$1$};
    \node[] at (0,-1.8)   {\small$3$};
    \node[] at (.75,-1.8)   {\small$4$};

    \node[] at (1.25,-1.5)   {\Large,};

    \filldraw  (2.5,0) circle (2pt);
    \filldraw  (1.75,-1.5) circle (2pt);
    \filldraw  (2.5,-1.5) circle (2pt);
    \filldraw  (3.25,-1.5) circle (2pt);
   
    \node[] at (3.75,-.75)   {\Huge$\rbrace$};

    \draw [thick](2.5,0) -- (1.75,-1.5);
    \draw [thick](2.5,0) -- (2.5,-1.5);
    \draw [thick](2.5,0) -- (3.25,-1.5);
   
    \node[] at (2.5,.3)   {\small$4$};
     \node[] at (1.75,-1.8)   {\small$1$};
    \node[] at (2.5,-1.8)   {\small$2$};
    \node[] at (3.25,-1.8)   {\small$3$};
    
\end{tikzpicture} 
and the Ferrers board 
$\lambda=(6,6,6,6,5,3)$. Let $T_{\lambda}=(2,6,5,1,3,4)$ be a transversal of $\lambda$ 
avoiding $P^1_{(4,1,2)}$. We first construct the encoding 
$w_{T_{\lambda}}=w_1w_2w_3w_4w_5w_6$ associated with $T_{\lambda}$, and then use this encoding to obtain a transversal $T'_{\lambda}$ that avoids $P^2_{(4,1,2)}$.

$T_{\lambda}=(2,6,5,1,3,4)=$ \scalebox{0.83}{
\ytableausetup{centertableaux} 

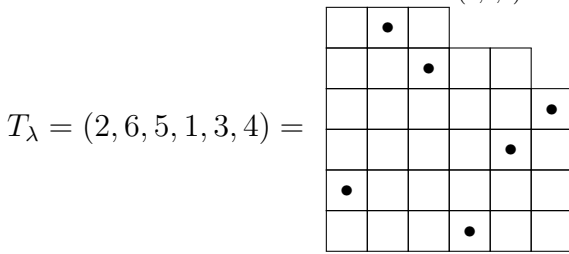
\begin{ytableau}
     \textcolor{white}{d} & \bullet & \textcolor{white}{d}\\
     \textcolor{white}{d} & \textcolor{white}{d} & \bullet & &\\
    \textcolor{white}{d} & &  & &\textcolor{white}{d}  &\bullet\\
     & \textcolor{white}{d} & \textcolor{white}{d} & &\bullet &\\
\bullet & \textcolor{white}{d} & \textcolor{white}{d} &  & & \\
      & \textcolor{white}{d} & \textcolor{white}{d} & \bullet & & 
\end{ytableau}}

\vspace{0.2cm}

In the first (topmost) row there are $l=3$ white cells. Since $l<m=4$, all 
three positions are valid; as $1$ lies in the second valid position, we get 
$w_1=2$. In the second row, $l=4$, but only two positions (the second and 
fourth columns) are valid; since $1$ occupies the first of these, $w_2=1$. 
Similarly, $w_3=2$. 

In the fourth row, $l=3<m$, so all positions are valid and $1$ appears in 
the third position, giving $w_4=3$. In the fifth row, there are two white 
cells and $1$ is in the first valid position, so $w_5=1$. The last row has 
only one possible position, hence $w_6=0$. 

Thus $w_{T_{\lambda}}=(2,1,2,3,1,0)$. The red entries indicate the transversal and corresponding encoding.

\begin{figure}[!htbp]

   \centering
{

\scalebox{0.83}{
\begin{ytableau}
     $1$  & $\red{2}$ & $3$\\
     \textcolor{white}{d} & \textcolor{white}{d} & \textcolor{white}{d} & &\\
    \textcolor{white}{d} & &  & &\textcolor{white}{d}  &\\
     & \textcolor{white}{d} & \textcolor{white}{d} & & &\\
     & \textcolor{white}{d} & \textcolor{white}{d} &  & & \\
      & \textcolor{white}{d} & \textcolor{white}{d} &  & & 
\end{ytableau}
$\longrightarrow$
\begin{ytableau}
     *(lgrey)$1$  & *(lgrey)$\red{2}$ & *(lgrey)$3$\\
     \textcolor{white}{d} &  *(lgrey) & $\red{1}$  & & $2$\\
    \textcolor{white}{d} &    *(lgrey) &  & &\textcolor{white}{d}  &\\
     &    *(lgrey) & \textcolor{white}{d} & & &\\
     &    *(lgrey) & \textcolor{white}{d} &  & &\\
      & *(lgrey) & \textcolor{white}{d} &  & & \\
\end{ytableau}
$\longrightarrow$
\begin{ytableau}
     *(lgrey)$1$  & *(lgrey)$\red{2}$ &*(lgrey) $3$\\
     *(lgrey) &  *(lgrey) & *(lgrey)$\red{1}$  & *(lgrey)& *(lgrey)$2$\\
    \textcolor{white}{d} &    *(lgrey) & *(lgrey) & $1$& & $\red{2}$\\
     &    *(lgrey) & *(lgrey) & & &\\
     &    *(lgrey) & *(lgrey) &  & &\\
      & *(lgrey) & *(lgrey) &  & & \\
\end{ytableau}
$\longrightarrow$
\begin{ytableau}
    *(lgrey) $1$  & *(lgrey)$\red{2}$ & *(lgrey)$3$\\
     *(lgrey) &  *(lgrey) & *(lgrey)$\red{1}$  & *(lgrey)& *(lgrey)$2$\\
    *(lgrey) &    *(lgrey) & *(lgrey) & *(lgrey) $1$& *(lgrey)& *(lgrey)$\red{2}$\\
     $1$ &    *(lgrey) & *(lgrey) & $2$ & $\red{3}$ &*(lgrey)\\
     &    *(lgrey) & *(lgrey) &  & &*(lgrey)\\
      & *(lgrey) & *(lgrey) &  & & *(lgrey)\\
\end{ytableau}

}\vspace{0.2cm}

\scalebox{0.835}{
$\longrightarrow$
\begin{ytableau}
     *(lgrey)$1$  & *(lgrey)$\red{2}$ & *(lgrey)$3$\\
     *(lgrey) &  *(lgrey) & *(lgrey)$\red{1}$  &*(lgrey) & *(lgrey)$2$\\
   *(lgrey) &    *(lgrey) & *(lgrey) & *(lgrey)$1$& *(lgrey)& *(lgrey)$\red{2}$\\
     *(lgrey)$1$ &    *(lgrey) & *(lgrey) &*(lgrey) $2$ & *(lgrey)$\red{3}$ &*(lgrey)\\
     $\red{1}$&    *(lgrey) & *(lgrey) & $2$ & *(lgrey) &*(lgrey)\\
      & *(lgrey) & *(lgrey) &   &*(lgrey) & *(lgrey)\\
\end{ytableau}
$\longrightarrow$
\begin{ytableau}
     *(lgrey)$1$  & *(lgrey)$\red{2}$ & *(lgrey)$3$\\
     *(lgrey) &  *(lgrey) & *(lgrey)$\red{1}$  &*(lgrey) & *(lgrey)$2$\\
   *(lgrey) &    *(lgrey) & *(lgrey) & *(lgrey)$1$& *(lgrey)& *(lgrey)$\red{2}$\\
     *(lgrey)$1$ &    *(lgrey) & *(lgrey) &*(lgrey) $2$ & *(lgrey)$\red{3}$ &*(lgrey)\\
     *(lgrey)$\red{1}$&    *(lgrey) & *(lgrey) & *(lgrey)$2$ & *(lgrey) &*(lgrey)\\
      *(lgrey)& *(lgrey) & *(lgrey) &  $\red{0}$ &*(lgrey) & *(lgrey)\\
\end{ytableau}
}
}
\caption{ Step-by-step insertion process for $T_{\lambda}$}  
\label{Fig2}
 \end{figure}

 We now construct the transversal avoiding $P^2_{(4,1,2)}$ from the encoding 
$w_{T_{\lambda}}=(2,1,2,3,1,0)$. The boards below illustrate the insertion process, which proceeds exactly as before. 

\begin{figure}[!htbp]

   \centering
{

\scalebox{0.83}{
\begin{ytableau}
     $1$  & $\red{2}$ & $3$\\
     \textcolor{white}{d} & \textcolor{white}{d} & \textcolor{white}{d} & &\\
    \textcolor{white}{d} & &  & &\textcolor{white}{d}  &\\
     & \textcolor{white}{d} & \textcolor{white}{d} & & &\\
     & \textcolor{white}{d} & \textcolor{white}{d} &  & & \\
      & \textcolor{white}{d} & \textcolor{white}{d} &  & & 
\end{ytableau}
$\longrightarrow$
\begin{ytableau}
     *(lgrey)$1$  & *(lgrey)$\red{2}$ & *(lgrey)$3$\\
     $\red{1}$ &  *(lgrey) & $2$  & & \\
    \textcolor{white}{d} &    *(lgrey) &  & &\textcolor{white}{d}  &\\
     &    *(lgrey) & \textcolor{white}{d} & & &\\
     &    *(lgrey) & \textcolor{white}{d} &  & &\\
      & *(lgrey) & \textcolor{white}{d} &  & & \\
\end{ytableau}
$\longrightarrow$
\begin{ytableau}
     *(lgrey)$1$  & *(lgrey)$\red{2}$ & *(lgrey)$3$\\
     *(lgrey)$\red{1}$ &  *(lgrey) & *(lgrey)$2$  & *(lgrey)& *(lgrey) \\
    *(lgrey) &    *(lgrey) & $1$ & & $\red{2}$  &\\
     *(lgrey) &    *(lgrey) & \textcolor{white}{d} & & &\\
     *(lgrey)&    *(lgrey) & \textcolor{white}{d} &  & &\\
      *(lgrey)& *(lgrey) & \textcolor{white}{d} &  & & \\
\end{ytableau}
$\longrightarrow$
\begin{ytableau}
     *(lgrey)$1$  & *(lgrey)$\red{2}$ & *(lgrey)$3$\\
     *(lgrey)$\red{1}$ &  *(lgrey) & *(lgrey)$2$  & *(lgrey)& *(lgrey) \\
    *(lgrey) &    *(lgrey) & *(lgrey) $1$ &*(lgrey) & *(lgrey)$\red{2}$  &*(lgrey)\\
     *(lgrey)&    *(lgrey) & $1$ & $2$ &*(lgrey) & $\red{3}$\\
     *(lgrey)&    *(lgrey) & \textcolor{white}{d} &  &*(lgrey) &\\
      *(lgrey)& *(lgrey) & \textcolor{white}{d} &  & *(lgrey)& \\
\end{ytableau}

}\vspace{0.2cm}

\scalebox{0.835}{
$\longrightarrow$
\begin{ytableau}
     *(lgrey)$1$  & *(lgrey)$\red{2}$ & *(lgrey)$3$\\
     *(lgrey)$\red{1}$ &  *(lgrey) & *(lgrey)$2$  & *(lgrey)& *(lgrey) \\
    *(lgrey) &    *(lgrey) & *(lgrey) $1$ &*(lgrey) & *(lgrey)$\red{2}$  &*(lgrey)\\
     *(lgrey)&    *(lgrey) & *(lgrey)$1$ &*(lgrey) $2$ &*(lgrey) & *(lgrey)$\red{3}$\\
      *(lgrey)& *(lgrey) & $\red{1}$ & $2$  & *(lgrey)& *(lgrey)\\
      *(lgrey) &    *(lgrey) & \textcolor{white}{d} & &*(lgrey) &*(lgrey)\\
\end{ytableau}
$\longrightarrow$
\begin{ytableau}
     *(lgrey)$1$  & *(lgrey)$\red{2}$ & *(lgrey)$3$\\
     *(lgrey)$\red{1}$ &  *(lgrey) & *(lgrey)$2$  & *(lgrey)& *(lgrey) \\
    *(lgrey) &    *(lgrey) & *(lgrey) $1$ &*(lgrey) & *(lgrey)$\red{2}$  &*(lgrey)\\
     *(lgrey)&    *(lgrey) & *(lgrey)$1$ &*(lgrey) $2$ &*(lgrey) & *(lgrey)$\red{3}$\\
     *(lgrey)& *(lgrey) & *(lgrey) $\red{1}$ & *(lgrey) $2$  & *(lgrey)& *(lgrey)\\
      *(lgrey)& *(lgrey) &*(lgrey)  & $\red{0}$ & *(lgrey)& *(lgrey)\\
\end{ytableau}
}
}
\caption{ Step-by-step insertion process for $T'_{\lambda}$}  
\label{Fig3}
 \end{figure}

Hence we obtain the transversal $T'_{\lambda}=(5,6,2,1,4,3)=$ \scalebox{0.83}{
\ytableausetup{centertableaux} 
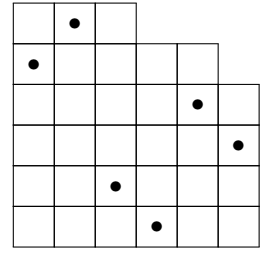
\begin{ytableau}
     \textcolor{white}{d} & \bullet & \textcolor{white}{d}\\
     \bullet & \textcolor{white}{d} &  & &\\
    \textcolor{white}{d} & &  & & \bullet &\\
     & \textcolor{white}{d} & \textcolor{white}{d} & &  & \bullet\\
 & \textcolor{white}{d} & \bullet &  & & \\
      & \textcolor{white}{d} & \textcolor{white}{d} & \bullet & & 
\end{ytableau}}







\begin{note}\label{bijection}
    From the above discussion it is clear that for a fixed Ferrers board $\lambda$ each transversal give us a unique encoding word. And each encoding word gives us a unique transversal in that Ferrers board $\lambda$. Therfore this gives us a bijection between $S_{\lambda}(P^a_{(m,k,d)})$ and set of all encoding words for the Ferrers board $\lambda$.
\end{note}

\begin{theorem}\label{size_I}
    The cardinality of $I$ is 
    $$m - kd - 1 + k \cdot \max(0, d - l + m - 1),$$ which is independent of the boundary parameter $a$ and only depends on $m,k,d,l$.
\end{theorem}

\begin{proof}

Let \(l,m,d,k,a\) be positive integers with \(a+kd \le l\). Define 
\(\Delta=l-m+1\) and consider the set
\[
I=\{x\in \mathbb{Z}^+ : x<a\}\ \cup\ 
\bigcup_{j=0}^{k}\{x: a+jd<x\le \text{upper}_j,\ 
l-x<m-(a+pd)\ \text{for all }p\le j\},
\]
where \(\text{upper}_j=a+(j+1)d-1\) for \(j<k\) and 
\(\text{upper}_k=l\).

To compute \(|I|\), partition the integers into three disjoint parts: 
\begin{enumerate}
     \item[(i)] \textbf{Base set:} The region \(1\le x<a\) contributes \(a-1\) elements. 
    \item[(ii)] \textbf{Intermediate gaps:}  For each \(j=0,1,\ldots ,k-1\), we have $a+jd<x<a+(j+1)d$ and $l-x<m-(a+jd)$. Combining these two inequalities the valid values lie in $a+jd+\Delta \le x \le a+(j+1)d-1$.
    The number of integers in this range is $ (a + jd + d - 1) - (a + jd + l - m + 1) + 1=d-\Delta$.
 To prevent negative counts when the required displacement exceeds the step size ($\Delta > d$), the cardinality per gap is bounded below by $0$, yielding $k \cdot \max(0, d - \Delta)$.
    \item[(iii)] \textbf{Terminal interval:} For \(j=k\), the range 
\(a+kd+\Delta \le x \le l\) contributes 
\(l-a-kd-\Delta+1\) elements.
\end{enumerate}
Summing these disjoint contributions gives
\begin{equation*}
|I| = \underbrace{(a - 1)}_{\text{Base}} + \underbrace{k \cdot \max(0, d - \Delta)}_{\text{Intermediate}} + \underbrace{(l - a - kd - \Delta + 1)}_{\text{Terminal}}
\end{equation*}
Expanding the summation:
\begin{align*}
|I| = (a - 1) + k \cdot \max(0, d - \Delta) + (l - a - kd - \Delta + 1)
\end{align*}
Substituting $\Delta = l - m + 1$ yields the closed-form expression:
\begin{equation*}
|I| = m - kd - 1 + k \cdot \max(0, d - l + m - 1)
\end{equation*}
Therefore the size of $I$ only depends on the parameters $m,k,d$ and $l$.
\end{proof}
\begin{note}
    Because the final expression for $|I|$ is entirely independent of $a$, it follows that for any alternative boundary parameter $b$ satisfying the same global constraints (e.g., $b > a$), the cardinality remains invariant.
\end{note}

\begin{corollary}
Let $a, d, k, m,$ and $l$ be positive integers satisfying the boundary condition $a + kd \le m$. The cardinality of the set $I$ is always non-negative, i.e., $|I| \ge 0$.
\end{corollary}
The proof is immediate from the expression of $|I|$. 






\begin{corollary}
The set $I$ is empty, i.e., $|I|=0$ if and only if 
$a=1$, $m=kd+1$, and $l\ge d(k+1)$.
\end{corollary}

\begin{proof}
From Theorem~\ref{size_I},
$$
|I|=(m-kd-1)+k\cdot \max(0,d-l+m-1),
$$
where both terms are non-negative.

Assume $|I|=0$. Then both terms must be zero. Hence
$
m-kd-1=0 \implies m=kd+1.
$
Using the condition $a+kd\le m$, we get
\[
a+kd\le kd+1 \implies a\le 1.
\]
Since $a\ge 1$, it follows that $a=1$.

Also,
$
\max(0,d-l+m-1)=0
$ implies
$
d-l+m-1\le 0.
$
Substituting $m=kd+1$, we obtain
\[
d-l+kd\le 0
\implies d(k+1)\le l.
\]

Conversely, if $a=1$, $m=kd+1$, and $l\ge d(k+1)$, then
\[
m-kd-1=0
\]
and
\[
d-l+m-1=d(k+1)-l\le 0.
\]
Thus,
\[
k\cdot \max(0,d-l+m-1)=0.
\]
Therefore $|I|=0$.
\end{proof}

\begin{note}
 At any stage of the encoding process, if we obtain $|I|=0$, then no transversal is valid for the corresponding Ferrers board $\lambda$.
\end{note}

\paragraph{Proof of Theorem \ref{main_thm} :}
At each step of the encoding process with \(l \ge m\) white cells, Theorem~\ref{size_I} implies that the number of valid positions is
\[
|I| = m - kd - 1 + k \cdot \max(0,\, d - l + m - 1).
\]
Notably, this quantity is independent of the boundary parameter \(a\). Hence, at every step of the encoding process, the sets \(P^a_{(m,k,d)}\) and \(P^b_{(m,k,d)}\) admit the same number of valid positions.

It follows that the total number of transversals of a fixed Ferrers board \(\lambda\) avoiding \(P^a_{(m,k,d)}\), obtained as the product of the cardinalities of \(I\) over all steps, is independent of \(a\). In particular,
\[
|S_{\lambda}(P^a_{(m,k,d)})| = |S_{\lambda}(P^b_{(m,k,d)})|.
\]

Moreover, the encoding process defines a bijection \(\phi_a\) between \(S_{\lambda}(P^a_{(m,k,d)})\) and the set of encoding words associated with \(\lambda\) (see Note \ref{bijection}). Similarly, we obtain a bijection \(\phi_b\) between \(S_{\lambda}(P^b_{(m,k,d)})\) and the same set of encoding words. Consequently, the composition \(\phi_b^{-1} \circ \phi_a\) yields a bijection between \(S_{\lambda}(P^a_{(m,k,d)})\) and \(S_{\lambda}(P^b_{(m,k,d)})\).

Therefore, for any distinct \(a,b \in [m]\),
$$
P^a_{(m,k,d)} \sim_s P^b_{(m,k,d)}.
$$
\qed

\begin{theorem}\label{claw_set_size}
If $n<m$ then $|S_n(P^a_{(m,k,d)})|=n!$ and if $n\geq m$ then
\begin{align*}
    |S_n(P^a_{(m,k,d)})|=& \left(\displaystyle \prod_{d+m-1< l\leq n}\{(m-kd-1)\}\right) \\
    & \displaystyle\left( \prod_{m\leq l \leq d+m-1}(m-kd-1)+k(d-l+m-1)\right ) \displaystyle\left(\prod_{1\leq l \leq m-1}l\right). 
\end{align*}
\end{theorem}
\begin{proof}
The result is immediate for $n<m$, since no pattern from 
$P^a_{(m,k,d)}$ can occur in a permutation of length $n$.

Now assume $n\ge m$. Considering the $n\times n$ square Ferrers board, at each step of the encoding process with $l\ge m$ white cells,  by Theorem \ref{size_I}, the number of valid positions is
\[
|I|=m-kd-1+k \cdot \max(0,d-l+m-1).
\]
If $l<m$, then all $l$ positions are valid. Taking the product over all steps of the encoding process gives the desired result.
\end{proof}

\begin{note}
 When we consider a single claw shape POP as in Figure \ref{Fig1}, i.e., $k=d=0$ then from Theorem \ref{claw_set_size} we have,  $|S_n(P^a_{(m,0,0)})|=(m-1)^{n-m+1}(m-1)!$  for $n \geq m$, which is exactly same as Theorem \ref{claw_size}.   
\end{note}

In conclusion, this paper has successfully broadened the understanding of partially ordered patterns by establishing the shape-Wilf-equivalence for extensive families of claw-shaped POPs. Central to our results is the development of a novel, recursive encoding process for transversals of Ferrers boards, which offers a fundamentally new bijective methodology distinct from prior literature. By rigorously defining the valid positional constraints within this framework, we not only proved the shape-Wilf-equivalence of these pattern sets but also derived a generalized formula to enumerate the permutations avoiding them. We anticipate that this encoding technique could be adapted in future research to explore shape-Wilf-equivalence and enumeration for other complex poset structures and generalized pattern families.

\end{document}